\documentclass[11pt]{amsart}%
\usepackage{stmaryrd}
\usepackage{amsfonts}
\usepackage{enumerate}
\usepackage{amsmath}
\usepackage{amssymb}
\usepackage{graphicx}
\usepackage{color}
\usepackage{vmargin}%
\providecommand{\U}[1]{\protect\rule{.1in}{.1in}}

\setpapersize{USletter}
\setmarginsrb{1in}{0.7in}{1in}{0.6in}{0in}{0in}{10pt}{0.4in}
\theoremstyle{plain}
\newtheorem{theorem}{Theorem}[section]
\newtheorem{corollary}[theorem]{Corollary}

\newtheorem{proposition}[theorem]{Proposition}

\newtheorem{definition}[theorem]{Definition}
\begin{document}
\title{Nonnegative approximations of nonnegative tensors}
\author[L.-H.~Lim]{Lek-Heng Lim}
\address{Department of Mathematics, University of California, Berkeley, CA 94720-3840}
\email{lekheng@math.berkeley.edu}
\author[P.~Comon]{Pierre Comon}
\address{Lab.~I3S, CNRS UMR6070, University of Nice, F-06903, Sophia-Antipolis, France}
\email{pcomon@unice.fr}

\begin{abstract}
We study the decomposition of a nonnegative tensor into a minimal sum of outer
product of nonnegative vectors and the associated parsimonious na\"{\i}ve
Bayes probabilistic model. We show that the corresponding approximation
problem, which is central to nonnegative \textsc{parafac}, will always have
optimal solutions. The result holds for any choice of norms and, under a mild
assumption, even Br\`{e}gman divergences.

\end{abstract}
\maketitle

\section{Dedication}

This article is dedicated to the memory of our late colleague Richard Allan
Harshman. It is loosely organized around two of Harshman's best known works
--- \textsc{parafac} \cite{Ha} and \textsc{lsi} \cite{LSI}, and answers two
questions that he posed. We target this article at a technometrics readership.

In Section \ref{sec:NTD}, we discussed a few aspects of nonnegative tensor
factorization and Hofmann's \textsc{plsi}, a variant of the \textsc{lsi} model
co-proposed by Harshman \cite{LSI}. In Section \ref{SectDegen}, we answered a
question of Harshman on why the apparently unrelated construction of Bini,
Capovani, Lotti, and Romani in \cite{BCLR} should be regarded as the first
example of what he called `\textsc{parafac} degeneracy' \cite{KHL}. Finally in
Section \ref{sec:Exist}, we showed that such \textsc{parafac} degeneracy will
not happen for nonnegative approximations of nonnegative tensors, answering
another question of his.

\section{Introduction\label{sec:Intro}}

The \textit{decomposition} of a tensor into a minimal sum of outer products of
vectors was first studied by Hitchcock \cite{Hi1, Hi2} in 1927. The topic has
a long and illustrious history in algebraic computational complexity theory
(cf.\ \cite{BCS} and the nearly 600 references in its bibliography) dating
back to Strassen's celebrated result \cite{Strass}. It has also recently found
renewed interests, coming most notably from algebraic statistics and quantum computing.

However the study of the corresponding \textit{approximation} problem,
i.e.\ the approximation of a tensor by a sum of outer products of vectors,
probably first surfaced as data analytic models in psychometrics\ in the work
of Harshman \cite{Ha}, who called his model \textsc{parafac} (for Parallel
Factor Analysis), and the work of Carrol and Chang \cite{CarrC70:psy}, who
called their model \textsc{candecomp} (for Canonical Decomposition).

The \textsc{candecomp}/\textsc{parafac} model, sometimes abbreviated as
\textsc{cp} model, essentially asks for a solution to the following problem:
given a tensor $A\in\mathbb{R}^{d_{1}\times\dots\times d_{k}}$, find an
optimal rank-$r$ approximation to $A$,%
\begin{equation}
X_{r}\in\operatorname*{argmin}\nolimits_{\operatorname*{rank}(X)\leq r}\lVert
A-X\rVert, \label{approx}%
\end{equation}
or, more precisely, find scalars $\lambda_{p}$ and unit
vectors\footnote{Whenever possible, we will use $\mathbf{u}_{p},\mathbf{v}%
_{p},\dots,\mathbf{z}_{p}$ instead of the more cumbersome $\mathbf{u}%
_{p}^{(1)},\mathbf{u}_{p}^{(2)},\dots,\mathbf{u}_{p}^{(k)}$ to denote the
vector factors in an outer product. It is to be understood that there are $k$
vectors in \textquotedblleft$\mathbf{u}_{p},\mathbf{v}_{p},\dots
,\mathbf{z}_{p}$,\textquotedblright\ where $k\geq3$.} $\mathbf{u}%
_{p},\mathbf{v}_{p},\dots,\mathbf{z}_{p}$, $p=1,\dots,r$, that minimizes%
\begin{equation}
\left\Vert A-\sum\nolimits_{p=1}^{r}\lambda_{p}\,\mathbf{u}_{p}\otimes
\mathbf{v}_{p}\otimes\dots\otimes\mathbf{z}_{p}\right\Vert . \label{parafac}%
\end{equation}
The norm $\lVert\,\cdot\,\rVert$ here is arbitrary and we will discuss several
natural choices in the next section. When $k=2$, $A$ becomes a matrix and a
solution to the problem when $\lVert\,\cdot\,\rVert$ is unitarily invariant is
given by the celebrated Eckart-Young theorem: $X_{r}$ may be taken to be%
\[
X_{r}=\sum\nolimits_{p=1}^{r}\sigma_{p}\,\mathbf{u}_{p}\otimes\mathbf{v}_{p},
\]
where $\sigma_{1}\geq\dots\geq\sigma_{r}$ are the first $r$ singular values of
$A$ and $\mathbf{u}_{p},\mathbf{v}_{p}$ the corresponding left and right
singular vectors.

However when $k\geq3$ the problem becomes more subtle. In fact, a global
minimizer of \eqref{parafac} may not even exist as soon as $k\geq3$; in which
case the problem in \eqref{approx} is ill-posed because the set of minimizers
is empty. We refer the reader to Section \ref{SectDegen} for examples and
discussions. Nevertheless we will show that for nonnegative tensors the
problem of finding a best nonnegative rank-$r$ approximation always has a
solution, i.e.\ \eqref{parafac} will always have a global minimum when $A$ and
$\mathbf{u}_{p},\mathbf{v}_{p},\dots,\mathbf{z}_{p}$ are required to be
nonnegative. Such nonnegativity arises naturally in applications. For example,
in the context of chemometrics, sample concentration and spectral intensity
often cannot assume negative values \cite{BdJ, BS, CDP, KtB, Paa1, Paa2}.
Nonnegativity can also be motivated by the data analytic tenet \cite{LS1} that
the way `basis functions' combine to build `target objects' is an exclusively
additive process and should not involve any cancellations between the basis
functions. For $k=2$, this is the motivation behind \textit{nonnegative matrix
factorization} (\textsc{nmf}) \cite{LS1, Paa2}, essentially a decomposition of
a nonnegative matrix $A\in\mathbb{R}^{m\times n}$ into a sum of outer-products
of nonnegative vectors,
\[
A=WH^{\top}=\sum\nolimits_{p=1}^{r}\mathbf{w}_{p}\otimes\mathbf{h}_{p},
\]
or, in the noisy situation, the approximation of a nonnegative matrix by such
a sum:
\[
\min\nolimits_{W\geq0,H\geq0}\lVert A-WH^{\top}\rVert=\min
\nolimits_{\mathbf{w}_{p}\geq0,\mathbf{h}_{p}\geq0}\left\Vert A-\sum
\nolimits_{p=1}^{r}\mathbf{w}_{p}\otimes\mathbf{h}_{p}\right\Vert .
\]
The generalization of \textsc{nmf} to tensors of higher order yields a model
known as \textit{nonnegative }\textsc{parafac} \cite{CDP, KtB, Paa1}, which
has also been studied more recently under the name \textit{nonnegative tensor
factorization} (\textsc{ntf}) \cite{NTF}. As we have just mentioned, a general
tensor can fail to have a best low-rank approximation. So the first question
that one should ask in a multilinear generalization of a bilinear model is
whether the generalized problem would still have a solution --- and this was
the question that Harshman posed.\ More generally, we will show that
nonnegative \textsc{parafac} always has a solution for any continuous measure
of proximity satisfying some mild conditions, e.g.\ norms or Br\`{e}gman
divergences. These include the sum-of-squares loss and Kullback-Leibler
divergence commonly used in \textsc{nmf} and \textsc{ntf}.

The following will be proved in Sections \ref{sec:Exist} and \ref{secBreg}.
Let $\Omega_{0}\subseteq\Omega\subseteq\mathbb{R}_{+}^{d_{1}\times\dots\times
d_{k}}$ be closed convex subsets. Let $d:\Omega\times\Omega_{0}\rightarrow
\mathbb{R}$ be a norm or a Br\`{e}gman divergence. For any nonnegative tensor
$A\in\Omega$ and any given $r\in\mathbb{N}$, a \textit{best nonnegative
rank-}$r$\textit{ approximation} always exist in the sense that the following
infimum
\[
\inf\{d\left(  A,X\right)  \mid X\in\Omega_{0}, \operatorname*{rank}%
\nolimits_{+}(X)\leq r \}
\]
is attained by some nonnegative tensor $X_{r}\in\Omega_{0}$,
$\operatorname*{rank}_{+}(X_{r})\leq r$. In particular, the nonnegative tensor
approximation problem%
\[
X_{r}\in\operatorname*{argmin}\nolimits_{\operatorname*{rank}_{+}(X)\leq
r}\lVert A-X\rVert
\]
is well-posed. Here $\operatorname*{rank}_{+}(X)$ denotes the
\textit{nonnegative rank} of $X$ and will be formally introduced in Section
\ref{sec:NTD}.

\section{Tensors as hypermatrices\label{sec:Ten}}

Let $V_{1},\dots,V_{k}$ be real vector spaces of dimensions $d_{1},\dots
,d_{k}$ respectively. An element of the tensor product $V_{1}\otimes
\dots\otimes V_{k}$ is called an order-$k$ \textit{tensor}. Up to a choice of
bases on $V_{1},\dots,V_{k}$, such a tensor may be represented by a
$d_{1}\times\dots\times d_{k}$ array of real numbers\footnote{The subscripts
and superscripts will be dropped when the range of $j_{1},\dots,j_{k}$ is
obvious or unimportant. We use double brackets to delimit hypermatrices.},
\begin{equation}
A=\llbracket a_{j_{1}\cdots j_{k}}\rrbracket_{j_{1},\dots,j_{k}=1}%
^{d_{1},\dots,d_{k}}\in\mathbb{R}^{d_{1}\times\dots\times d_{k}}.
\label{tenspace}%
\end{equation}
Gelfand, Kapranov, and Zelevinsky called such coordinate representations of
abstract tensors \textit{hypermatrices} \cite{GKZ1}. It is worth pointing out
that an \textit{array} is just a data structure but like matrices,
hypermatrices are more than mere arrays of numerical values. They are equipped
with algebraic operations arising from the algebraic structure of
$V_{1}\otimes\dots\otimes V_{k}$:

\begin{itemize}
\item \textit{Addition and Scalar Multiplication}: For $\llbracket
a_{j_{1}\cdots j_{k}}\rrbracket,\llbracket b_{j_{1}\cdots j_{k}}%
\rrbracket\in\mathbb{R}^{d_{1}\times\dots\times d_{k}}$ and $\lambda,\mu
\in\mathbb{R}$,
\begin{equation}
\lambda\llbracket a_{j_{1}\cdots j_{k}}\rrbracket+\mu\llbracket b_{j_{1}\cdots
j_{k}}\rrbracket=\llbracket\lambda a_{j_{1}\cdots j_{k}}+\mu b_{j_{1}\cdots
j_{k}}\rrbracket\in\mathbb{R}^{d_{1}\times\dots\times d_{k}}. \label{prop1}%
\end{equation}

\item \textit{Outer Product Decomposition}: Every $A=\llbracket a_{j_{1}\cdots
j_{k}}\rrbracket\in\mathbb{R}^{d_{1}\times\dots\times d_{k}}$ may be
decomposed as%
\begin{equation}
A=\sum\nolimits_{p=1}^{r}\lambda_{p}\,\mathbf{u}_{p}\otimes\mathbf{v}%
_{p}\otimes\dots\otimes\mathbf{z}_{p},\qquad a_{j_{1}\cdots j_{k}}%
=\sum\nolimits_{p=1}^{r}\lambda_{p}u_{pj_{1}}v_{pj_{2}}\cdots z_{pj_{k}},
\label{prop2a}%
\end{equation}
with $\lambda_{p}\in\mathbb{R}$, $\mathbf{u}_{p}=[u_{p1},\dots,u_{pd_{1}%
}]^{\top}\in\mathbb{R}^{d_{1}},\dots,\mathbf{z}_{p}=[z_{p1},\dots,z_{pd_{k}%
}]^{\top}\in\mathbb{R}^{d_{k}}$, $p=1,\dots,r$.
\end{itemize}

The symbol $\otimes$ denotes the \textit{Segre outer product}: For vectors
$\mathbf{x}=[x_{1},\dots,x_{l}]^{\top}\in\mathbb{R}^{l}$, $\mathbf{y}%
=[y_{1},\dots,y_{m}]^{\top}\in\mathbb{R}^{m}$, $\mathbf{z}=[z_{1},\dots
,z_{n}]^{\top}\in\mathbb{R}^{n}$, the quantity $\mathbf{x}\otimes
\mathbf{y}\otimes\mathbf{z}$, is simply the $3$-hypermatrix $\llbracket
x_{i}y_{j}z_{k}\rrbracket_{i,j,k=1}^{l,m,n}\in\mathbb{R}^{l\times m\times n}$,
with obvious generalization to an arbitrary number of vectors.

It follows from \eqref{prop1} that $\mathbb{R}^{d_{1}\times\dots\times d_{k}}$
is a vector space of dimension $d_{1}\cdots d_{k}$. The existence of a
decomposition \eqref{prop2a} distinguishes $\mathbb{R}^{d_{1}\times\dots\times
d_{k}}$ from being merely a vector space by endowing it with a tensor product
structure. While as real vector spaces, $\mathbb{R}^{l\times m\times n}$
(hypermatrices), $\mathbb{R}^{lm\times n},\mathbb{R}^{ln\times m}%
,\mathbb{R}^{mn\times l}$ (matrices), and $\mathbb{R}^{lmn}$ (vectors) are all
isomorphic, the tensor product structure distinguishes them. Note that a
different choice of bases on $V_{1},\dots,V_{k}$ would lead to a different
hypermatrix representation of elements in $V_{1}\otimes\dots\otimes V_{k}$. So
strictly speaking, a tensor and a hypermatrix are different in the same way a
linear operator and a matrix are different. Furthermore, just as a bilinear
functional, a linear operator, and a dyad may all be represented by the same
matrix, different types of tensors may be represented by the same hypermatrix
if one disregards covariance and contravariance. Nonetheless the term `tensor'
has been widely used to mean a hypermatrix in the data analysis communities
(including bioinformatics, computer vision, machine learning,
neuroinformatics, pattern recognition, signal processing, technometrics), and
we will refrain from being perverse and henceforth adopt this naming
convention. For the more pedantic readers, it is understood that what we call
a tensor in this article really means a hypermatrix.

A non-zero tensor that can be expressed as an outer product of vectors is
called a rank-1 tensor. More generally, the \textit{rank} of a tensor
$A=\llbracket a_{j_{1}\cdots j_{k}}\rrbracket_{j_{1},\dots,j_{k}=1}%
^{d_{1},\dots,d_{k}}\in\mathbb{R}^{d_{1}\times\dots\times d_{k}}$, denoted
$\operatorname*{rank}(A)$, is defined as the minimum $r$ for which $A$ may be
expressed as a sum of $r$ rank-1 tensors \cite{Hi1, Hi2},%
\begin{equation}
\operatorname*{rank}(A):=\min\Bigl\{r\Bigm|A=\sum\nolimits_{p=1}^{r}%
\lambda_{p}\,\mathbf{u}_{p}\otimes\mathbf{v}_{p}\otimes\dots\otimes
\mathbf{z}_{p}\Bigr\}. \label{cp2}%
\end{equation}
The definition of rank in \eqref{cp2} agrees with the definition of matrix
rank when applied to an order-$2$ tensor.

The \textit{Frobenius norm} or $F$-\textit{norm} of a tensor $A=\llbracket
a_{j_{1}\cdots j_{k}}\rrbracket_{j_{1},\dots,j_{k}=1}^{d_{1},\dots,d_{k}}%
\in\mathbb{R}^{d_{1}\times\dots\times d_{k}}$ is defined by%
\begin{equation}
\lVert A\rVert_{F}=\Bigl[\sum\nolimits_{j_{1},\dots,j_{k}=1}^{d_{1}%
,\dots,d_{k}}\lvert a_{j_{1}\cdots j_{k}}\rvert^{2}\Bigr]^{\frac{1}{2}}.
\label{Frob}%
\end{equation}
The $F$-norm is by far the most popular choice of norms for tensors in data
analytic applications. However when $A$ is nonnegative valued, then there is a
more natural norm that allows us to interpret the normalized values of $A$ as
probability distribution values, as we will see in the next section. With this
in mind, we define the $E$-\textit{norm} and $G$-\textit{norm} by%
\begin{equation}
\lVert A\rVert_{E}=\sum\nolimits_{i_{1},\dots,i_{k}=1}^{d_{1},\dots,d_{k}%
}\lvert a_{j_{1}\cdots j_{k}}\rvert\label{Delta}%
\end{equation}
and%
\[
\lVert A\rVert_{G}=\max\{\lvert a_{j_{1}\cdots j_{k}}\rvert\mid j_{1}%
=1,\dots,d_{1};\dots;j_{k}=1,\dots,d_{k}\}.
\]
Observe that the $E$-, $F$-, and $G$-norms of a tensor $A$ are simply the
$l^{1}$-, $l^{2}$-, and $l^{\infty}$-norms of $A$ regarded as a vector of
dimension $d_{1}\cdots d_{k}$. Furthermore they are multiplicative on rank-$1$
tensors in the following sense:%
\begin{align}
\lVert\mathbf{u}\otimes\mathbf{v}\otimes\dots\otimes\mathbf{z}\rVert_{E}  &
=\lVert\mathbf{u}\rVert_{1}\lVert\mathbf{v}\rVert_{1}\cdots\lVert
\mathbf{z}\rVert_{1},\label{norm1}\\
\lVert\mathbf{u}\otimes\mathbf{v}\otimes\dots\otimes\mathbf{z}\rVert_{F}  &
=\lVert\mathbf{u}\rVert_{2}\lVert\mathbf{v}\rVert_{2}\cdots\lVert
\mathbf{z}\rVert_{2},\nonumber\\
\lVert\mathbf{u}\otimes\mathbf{v}\otimes\dots\otimes\mathbf{z}\rVert_{G}  &
=\lVert\mathbf{u}\rVert_{\infty}\lVert\mathbf{v}\rVert_{\infty}\cdots
\lVert\mathbf{z}\rVert_{\infty}.\nonumber
\end{align}
The $F$-norm has the advantage of being induced by an inner product on
$\mathbb{R}^{d_{1}\times\dots\times d_{k}}$, namely,%
\begin{equation}
\langle A,B\rangle=\sum\nolimits_{j_{1},\dots,j_{k}=1}^{d_{1},\dots,d_{k}%
}a_{j_{1}\cdots j_{k}}b_{j_{1}\cdots j_{k}}. \label{ip}%
\end{equation}
As usual, it is straightforward to deduce a Cauchy-Schwarz inequality%
\[
\lvert\langle A,B\rangle\rvert\leq\lVert A\rVert_{F}\lVert B\rVert_{F},
\]
and a H\"{o}lder inequality%
\[
\lvert\langle A,B\rangle\rvert\leq\lVert A\rVert_{E}\lVert B\rVert_{G}.
\]
Many other norms may be defined on a space of tensors. For any $1\leq
p\leq\infty$, one may define the $l^{p}$-equivalent of \eqref{Frob}, of which
$E$-, $F$-, and $G$-norms are special cases. Another common class of tensor
norms generalizes operator norms of matrices: For example if $A=\llbracket
a_{ijk}\rrbracket\in\mathbb{R}^{l\times m\times n}$ and%
\[
A(\mathbf{x},\mathbf{y},\mathbf{z}):=\sum\nolimits_{i,j,k=1}^{l,m,n}%
a_{ijk}x_{i}y_{j}z_{j_{k}}%
\]
denotes the associated trilinear functional, then%
\[
\lVert A\rVert_{p,q,r}:=\sup_{\mathbf{x},\mathbf{y},\mathbf{z}\neq\mathbf{0}%
}\frac{\lvert A(\mathbf{x},\mathbf{y},\mathbf{z})\rvert}{\lVert\mathbf{x}%
\rVert_{p}\lVert\mathbf{y}\rVert_{q}\lVert\mathbf{z}\rVert_{r}}%
\]
defines a norm for any $1\leq p,q,r\leq\infty$. Nevertheless all these norms
are equivalent (and thus induce the same topology) since the tensor product
spaces here are finite-dimensional. In particular, the results in this paper
apply to any choice of norms since they pertain to the convergence of
sequences of tensors.

The discussion in this section remains unchanged if $\mathbb{R}$ is replaced
by $\mathbb{C}$ throughout (apart from a corresponding replacement of the
Euclidean inner product in \eqref{ip} by the Hermitian inner product) though a
minor caveat is that the tensor rank as defined in \eqref{cp2} depends on the
choice of base fields (see \cite{dSL} for a discussion).

\section{Nonnegative decomposition of nonnegative tensors\label{sec:NTD}}

We will see that a finite collection of discrete random variables satisfying
both the na\"{\i}ve Bayes hypothesis and the Ockham principle of parsimony
have a joint probability distribution that, when regarded as a nonnegative
tensor on the probability simplex, decomposes in a nonnegative rank-revealing
manner that parallels the matrix singular value decomposition. This
generalizes Hofmann's probabilistic variant \cite{Hoff} of latent semantic
indexing (\textsc{lsi}), a well-known technique in natural language processing
and information retrieval that Harshman played a role in developing
\cite{LSI}. Nonnegative tensor decompositions were first studied in the
context of \textsc{parafac} with nonnegativity constraints by the
technometrics communities \cite{BdJ, BS, CDP, KtB, Paa1}. The interpretation
as a na\"{\i}ve Bayes decomposition of probability distributions into
conditional distributions was due to Garcia, Stillman, and Sturmfels
\cite{GSS} and Sashua and Hazan \cite{NTF}. It is perhaps worth taking this
opportunity to point out a minor detail that had somehow been neglected in
\cite{GSS, NTF}: the na\"{\i}ve Bayes hypothesis is not sufficient to
guarantee a nonnegative rank-revealing decomposition, one also needs the
Ockham principle of parsimony, i.e.\ the hidden variable in question has to be
minimally supported.

A tensor $A=\llbracket a_{j_{1}\cdots j_{k}}\rrbracket_{j_{1},\dots,j_{k}%
=1}^{d_{1},\dots,d_{k}}\in\mathbb{R}^{d_{1}\times\dots\times d_{k}}$ is
\textit{nonnegative}, denoted $A\geq0$, if all $a_{j_{1}\cdots j_{k}}\geq0$.
We will write $\mathbb{R}_{+}^{d_{1}\times\dots\times d_{k}}:=\{A\in
\mathbb{R}^{d_{1}\times\dots\times d_{k}}\mid A\geq0\}$. For $A\geq0$, a
\textit{nonnegative outer-product decomposition} is one of the form%
\begin{equation}
A=\sum\nolimits_{p=1}^{r}\delta_{p}\,\mathbf{u}_{p}\otimes\mathbf{v}%
_{p}\otimes\dots\otimes\mathbf{z}_{p} \label{eqNTD}%
\end{equation}
where $\delta_{p}\geq0$ and $\mathbf{u}_{p},\mathbf{v}_{p},\dots
,\mathbf{z}_{p}\geq0$ for $p=1,\dots,r$. It is clear that such a decomposition
exists for any $A\geq0$. The minimal $r$ for which such a decomposition is
possible will be called the \textit{nonnegative rank}. For $A\geq0$, this is
denoted and defined via%
\[
\operatorname*{rank}\nolimits_{+}(A):=\min\Bigl\{r\Bigm|A=\sum\nolimits_{p=1}%
^{r}\delta_{p}\,\mathbf{u}_{p}\otimes\mathbf{v}_{p}\otimes\dots\otimes
\mathbf{z}_{p},\quad\delta_{p},\mathbf{u}_{p},\mathbf{v}_{p},\dots
,\mathbf{z}_{p}\geq0\text{ for all }p\Bigr\}.
\]

Let $\Delta^{d}$ denote the \textit{unit }$d$-\textit{simplex}, i.e.\ the
convex hull of the standard basis vectors in $\mathbb{R}^{d+1}$. Explicitly,%
\[
\Delta^{d}:=\Bigl\{\sum\nolimits_{p=1}^{d+1}\delta_{p}\mathbf{e}_{p}%
\in\mathbb{R}^{d+1}\Bigm|\sum\nolimits_{p=1}^{d+1}\delta_{p}=1,\quad\delta
_{1},\dots,\delta_{d+1}\geq0\Bigr\}=\{\mathbf{x}\in\mathbb{R}_{+}^{d+1}%
\mid\lVert\mathbf{x}\rVert_{1}=1\}.
\]
For nonnegative valued tensors, the $E$-norm has the advantage that
\eqref{Delta} reduces to a simple sum of all entries. This simple observation
leads to the following proposition stating that the decomposition in
\eqref{eqNTD} may be realized over unit simplices if we normalize $A$ by its
$E$-norm.

\begin{proposition}
\label{prop:nntd} Let $A\in\mathbb{R}_{+}^{d_{1}\times\dots\times d_{k}}$ be a
nonnegative tensor with $\operatorname*{rank}_{+}(A)=r$. Then there exist
$\boldsymbol{\delta}=[\delta_{1},\dots,\delta_{r}]^{\top}\in\mathbb{R}_{+}%
^{r}$, $\mathbf{u}_{p}\in\mathbb{R}_{+}^{d_{1}-1},\mathbf{v}_{p}\in
\mathbb{R}_{+}^{d_{2}-1},\dots,\mathbf{z}_{p}\in\mathbb{R}_{+}^{d_{k}-1}$,
$p=1,\dots,r$, where%
\[
\lVert\boldsymbol{\delta}\rVert_{1}=\lVert A\rVert_{E}%
\]
and%
\[
\lVert\mathbf{u}_{p}\rVert_{1}=\lVert\mathbf{v}_{p}\rVert_{1}=\dots
=\lVert\mathbf{z}_{p}\rVert_{1}=1,
\]
such that%
\begin{equation}
A=\sum\nolimits_{p=1}^{r}\delta_{p}\,\mathbf{u}_{p}\otimes\mathbf{v}%
_{p}\otimes\dots\otimes\mathbf{z}_{p}. \label{ud}%
\end{equation}

\end{proposition}

\begin{proof}
If $A=0$, this is obvious. So we will suppose that $A\neq0$. By the minimality
of $r=\operatorname*{rank}_{+}(A)$, we know that $\mathbf{u}_{p}%
,\mathbf{v}_{p},\dots,\mathbf{z}_{p}$ in \eqref{ud} are all nonzero and we may
assume that%
\[
\lVert\mathbf{u}_{p}\rVert_{1}=\lVert\mathbf{v}_{p}\rVert_{1}=\dots
=\lVert\mathbf{z}_{p}\rVert_{1}=1
\]
since otherwise we may normalize%
\[
\mathbf{\hat{u}}_{p}=\mathbf{u}_{p}/\lVert\mathbf{u}_{p}\rVert_{1}%
,\mathbf{\hat{v}}_{p}=\mathbf{v}_{p}/\lVert\mathbf{v}_{p}\rVert_{1}%
,\dots,\mathbf{\hat{z}}_{p}=\mathbf{z}_{p}/\lVert\mathbf{z}_{p}\rVert_{1},
\]
and set%
\[
\hat{\delta}_{p}=\delta_{p}\lVert\mathbf{u}_{p}\rVert_{1}\lVert\mathbf{v}%
_{p}\rVert_{1}\cdots\lVert\mathbf{z}_{p}\rVert_{1},
\]
and still have an equation of the form in \eqref{ud}. It remains to show that%
\[
\lVert\boldsymbol{\delta}\rVert_{1}=\lVert A\rVert_{E}.
\]
Note that since all quantities involved are nonnegative,%
\[
\lVert A\rVert_{E}=\left\Vert \sum\nolimits_{p=1}^{r}\delta_{p}\,\mathbf{u}%
_{p}\otimes\mathbf{v}_{p}\otimes\dots\otimes\mathbf{z}_{p}\right\Vert
_{E}=\sum\nolimits_{p=1}^{r}\delta_{p}\lVert\mathbf{u}_{p}\otimes
\mathbf{v}_{p}\otimes\dots\otimes\mathbf{z}_{p}\rVert_{E}.
\]
By \eqref{norm1}, the \textsc{rhs} can be further simplified to%
\[
\sum\nolimits_{p=1}^{r}\delta_{p}\lVert\mathbf{u}_{p}\rVert_{1}\lVert
\mathbf{v}_{p}\rVert_{1}\cdots\lVert\mathbf{z}_{p}\rVert_{1}=\sum
\nolimits_{p=1}^{r}\delta_{p}=\lVert\boldsymbol{\delta}\rVert_{1},
\]
as required.
\end{proof}

Note that the conditions on the vectors imply that they lie in unit simplices
of various dimensions:%
\begin{equation}
\mathbf{u}_{1},\dots,\mathbf{u}_{r}\in\Delta^{d_{1}-1},\quad\mathbf{v}%
_{1},\dots,\mathbf{v}_{r}\in\Delta^{d_{2}-1},\quad\dots,\quad\mathbf{z}%
_{1},\dots,\mathbf{z}_{r}\in\Delta^{d_{k}-1}. \label{sim}%
\end{equation}

For $k=2$, the above decomposition is best viewed as a parallel to the
singular value decomposition of a matrix $A\in\mathbb{R}^{m\times n}$, which
is in particular an expression of the form%
\begin{equation}
A=\sum\nolimits_{p=1}^{r}\sigma_{p}\,\mathbf{u}_{p}\otimes\mathbf{v}_{p},
\label{svd}%
\end{equation}
where $r=\operatorname*{rank}(A)$,%
\[
\lVert\boldsymbol{\sigma}\rVert_{2}=\Bigl[\sum\nolimits_{p=1}^{r}\lvert
\sigma_{p}\rvert^{2}\Bigr]^{\frac{1}{2}}=\lVert A\rVert_{F},\quad
\text{and\quad}\lVert\mathbf{u}_{p}\rVert_{2}=\lVert\mathbf{v}_{p}\rVert
_{2}=1,
\]
for all $p=1,\dots,r$. Here $\boldsymbol{\sigma}=[\sigma_{1},\dots,\sigma
_{r}]^{\top}\in\mathbb{R}^{r}$ is the vector of nonzero singular values of
$A$. If $A$ is normalized to have unit $F$-norm, then all quantities in
\eqref{svd} may be viewed as living in unit spheres of various dimensions:
$A\in\mathbb{S}^{mn-1}$, $\boldsymbol{\sigma}\in\mathbb{S}^{r-1}$,
$\mathbf{u}_{1},\dots,\mathbf{u}_{r}\in\mathbb{S}^{m-1}$, $\mathbf{v}%
_{1},\dots,\mathbf{v}_{r}\in\mathbb{S}^{n-1}$ where $\mathbb{S}^{d-1}%
=\{\mathbf{x}\in\mathbb{R}^{d}\mid\lVert\mathbf{x}\rVert_{2}=1\}$ is the unit
sphere in $\mathbb{R}^{d}$. For $k=2$, the nonnegative matrix decomposition in
Proposition \ref{prop:nntd} is one where the unit spheres are replaced by unit
simplices and the $l^{2}$- and $F$-norms replaced by the $l^{1}$- and
$E$-norms. An obvious departure from the case of \textsc{svd} is that the
vectors in \eqref{sim} are not orthogonal.

Henceforth when we use the terms \textsc{ntf} and \textsc{nmf}, we will mean a
decomposition of the type in Proposition \ref{prop:nntd}. For a nonnegative
tensor with unit $E$-norm, $A\in\Delta^{d_{1}\cdots d_{k}-1}$, the
decomposition in Proposition \ref{prop:nntd} has a probabilistic interpretation.

Let $U,V,\dots,Z$ be discrete random variables and $q(u,v,\dots,z)=\Pr
(U=u,V=v,\dots,Z=z)$ be their joint probability distribution. Suppose
$U,V,\dots,Z$ satisfy the \textit{na\"{\i}ve Bayes hypothesis}, i.e.\ they are
conditionally independent upon a single hidden random variable $\Theta$. Let
$q_{1}(u\mid\theta),q_{2}(v\mid\theta),\dots,q_{k}(z\mid\theta)$ denote
respectively the marginal probability distributions of $U,V,\dots,Z$
conditional on the event $\Theta=\theta$. Then the probability distributions
must satisfy the relation
\begin{equation}
q(u,v,\dots,z)=\sum\nolimits_{\theta=1}^{r}\delta(\theta)\,q_{1}(u\mid
\theta)q_{2}(v\mid\theta)\cdots q_{k}(z\mid\theta) \label{eq:NB}%
\end{equation}
where $\delta(\theta)=\Pr(\Theta=\theta)$. Since the discrete random variables
$U,V,\dots,Z$ may take $d_{1},d_{2},\dots,d_{k}$ possible values respectively,
the Bayes rule in \eqref{eq:NB} can be rewritten as the tensor decomposition
in \eqref{ud}, provided we `store' the marginal distributions $q_{1}%
(u\mid\theta),q_{2}(v\mid\theta),\dots,q_{k}(z\mid\theta)$ in the vectors
$\mathbf{u}_{\theta},\mathbf{v}_{\theta},\dots,\mathbf{z}_{\theta}$
respectively. The requirement that $r=\operatorname*{rank}_{+}(A)$ corresponds
to the \textit{Ockham principle of parsimony}: that the model \eqref{eq:NB} be
the simplest possible, i.e.\ the hidden variable $\Theta$ be minimally supported.

For the case $k=2$, \eqref{eq:NB} is Hofmann's \textsc{plsi} \cite{Hoff}, a
probabilistic variant of \textit{latent semantic indexing} \cite{LSI}. While
it is known \cite{GG} that the multiplicative updating rule for \textsc{nmf}
with \textsc{kl} divergence in \cite{LS1} is equivalent to the use of
\textsc{em} algorithm for maximum likelihood estimation of \textsc{plsi} in
\cite{Hoff}, this is about the equivalence of two algorithms (\textsc{em} and
multiplicative updating) applied to two approximation problems (maximum
likelihood of \textsc{plsi} and minimum \textsc{kl} divergence of
\textsc{nmf}). Since the \textsc{em} algorithm and the \textsc{nmf}
multiplicative updating rules are first-order methods that can at best
converge to a stationary point, saying that these two algorithms are
equivalent for their respective approximation problems does not imply that the
respective models are equivalent. The preceding paragraph states that the
probabilistic relational models behind \textsc{plsi} and \textsc{ntf} (and
therefore \textsc{nmf}) are one and the same --- a collection of random
variables satisfying the na\"{\i}ve Bayes assumption with respect to a
parsimonious hidden variable. This is a statement independent of approximation
or computation.

\section{Nonexistence of globally optimal solution for real and complex tensor
approximations\label{SectDegen}}

A major difficulty that one should be aware of is that the problem of finding
a best rank-$r$ approximation for tensors of order $3$ or higher has no
solution in general. There exists $A\in\mathbb{R}^{d_{1}\times\dots\times
d_{k}}$ such that%
\begin{equation}
\inf\left\Vert A-\sum\nolimits_{p=1}^{r}\lambda_{p}\,\mathbf{u}_{p}%
\otimes\mathbf{v}_{p}\otimes\dots\otimes\mathbf{z}_{p}\right\Vert \label{inf}%
\end{equation}
is not attained by \textit{any} choice of $\lambda_{p},\mathbf{u}%
_{p},\mathbf{v}_{p},\dots,\mathbf{z}_{p}$, $p=1,\dots,r$. It is also in
general not possible to determine a priori if a given $A\in\mathbb{R}%
^{d_{1}\times\dots\times d_{k}}$ will fail to have a best rank-$r$
approximation. This problem is more widespread than one might imagine. It has
been shown in \cite{dSL} that examples of this failure happens over a wide
range of dimensions, orders, ranks, and for any continuous measure of
proximity (thus including all norms and Br\`{e}gman divergence). Moreover such
failures can occur with positive probability and in some cases with certainty,
i.e.\ where the infimum in \eqref{inf} is \textit{never} attained. This
phenomenon also extends to symmetric tensors \cite{CGLM}.

This poses some serious conceptual difficulties --- if one cannot guarantee a
solution a priori, then what is one trying to compute in instances where there
are no solutions? We often get the answer \textquotedblleft an approximate
solution\textquotedblright. But how could one approximate a solution that does
not even exist in the first place? Conceptual issues aside, this also causes
computational difficulties in practice. Forcing a solution in finite precision
for a problem that does not have a solution is an ill-advised strategy since a
well-posed problem near to an ill-posed one is, by definition, ill-conditioned
and therefore hard to compute. This ill-conditioning manifests itself in
iterative algorithms as summands that grew unbounded in magnitude but with the
peculiar property that the sum remains bounded. This was first observed by
Bini, Lotti, and Romani \cite{BLR} in the context of arbitrary precision
approximations (where this phenomenon is desirable). Independently Harshman
and his collaborators Kruskal and Lundy \cite{KHL} also investigated this
phenomenon, which they called \textsc{parafac} degeneracy, from the
perspective of model fitting (where it is undesirable). In Theorem \ref{thmE},
we will prove the cheerful fact that one does not need to worry about
\textsc{parafac} degeneracy when fitting a nonnegative \textsc{parafac} model.

The first published account of an explicitly constructed example of
\textsc{parafac} degeneracy appeared in a study of the complexity of matrix
multiplication by Bini, Capovani, Lotti, and Romani \cite{BCLR}. However their
discussion was for a context entirely different from data analysis/model
fitting and was presented in notations somewhat unusual. Until today, many
remain unconvinced that the construction in \cite{BCLR} indeed provides an
explicit example of \textsc{parafac} degeneracy and continue to credit the
much later work of Paatero \cite{Paa}. The truth is that such constructions
are well-known in algebraic computational complexity; in addition to
\cite{BCLR}, one may also find them in \cite{BLR, BCS, Kn}, all predating
\cite{Paa}. As a small public service\footnote{And also to fulfill, belatedly,
an overdued promise made to Richard when he was preparing his bibliography on
\textsc{parafac} degeneracy.}, we will translate the original construction of
Bini, Capovani, Lotti, and Romani into notations more familiar to the
technometrics communities.

In \cite{BCLR}, Bini, Capovani, Lotti, and Romani gave an algorithm that can
approximate to arbitrary precision the product of two $n\times n$ matrices and
requires only $O(n^{2.7799})$ scalar multiplications. The key to their
construction is the following triplet of matrices which at first glance seem
somewhat mysterious:%

\[
U=%
\begin{bmatrix}
1 & 0 & 1 & 0 & 1\\
0 & 0 & 0 & \varepsilon & \varepsilon\\
1 & 1 & 0 & 1 & 0
\end{bmatrix}
,\quad V=%
\begin{bmatrix}
\varepsilon & 0 & 0 & -\varepsilon & 0\\
0 & -1 & 0 & 1 & 0\\
0 & 0 & 0 & 0 & \varepsilon\\
1 & -1 & 1 & 0 & 1
\end{bmatrix}
,\quad W=%
\begin{bmatrix}
\varepsilon^{-1} & \varepsilon^{-1} & -\varepsilon^{-1} & \varepsilon^{-1} &
0\\
0 & 0 & 0 & 1 & 0\\
0 & 0 & -\varepsilon^{-1} & 0 & \varepsilon^{-1}\\
1 & 0 & 0 & 0 & -1
\end{bmatrix}
.
\]
We will show that these matrices may be used to construct a sequence of
tensors exhibiting \textsc{parafac} degeneracy.

We will assume that $U$ has a $4$th row of zeros and so $U,V,W\in
\mathbb{R}^{4\times4}$. As usual, $u_{ij},v_{ij},w_{ij}$ will denote the
$(i,j)$th entry of the respective matrices. Let $n\geq4$ and $\mathbf{x}%
_{1},\mathbf{x}_{2},\mathbf{x}_{3},\mathbf{x}_{4}\in\mathbb{R}^{n}$ (or
$\mathbb{C}^{n}$) be linearly independent vectors. For $\varepsilon>0$, define%
\[
A_{\varepsilon}:=\sum\nolimits_{j=1}^{4}\Bigl[\Bigl(\sum\nolimits_{i=1}%
^{4}u_{ij}\mathbf{x}_{i}\Bigr)\otimes\Bigl(\sum\nolimits_{i=1}^{4}%
v_{ij}\mathbf{x}_{i}\Bigr)\otimes\Bigl(\sum\nolimits_{i=1}^{4}w_{ij}%
\mathbf{x}_{i}\Bigr)\Bigr].
\]
Observe that%
\begin{multline*}
A_{\varepsilon}=(\mathbf{x}_{1}+\mathbf{x}_{3})\otimes(\varepsilon
\mathbf{x}_{1}+\mathbf{x}_{4})\otimes(\varepsilon^{-1}\mathbf{x}%
_{1}+\mathbf{x}_{4})+\mathbf{x}_{3}\otimes(-\mathbf{x}_{2}-\mathbf{x}%
_{4})\otimes\varepsilon^{-1}\mathbf{x}_{1}\\
+\mathbf{x}_{1}\otimes\mathbf{x}_{4}\otimes(-\varepsilon^{-1}\mathbf{x}%
_{1}-\varepsilon^{-1}\mathbf{x}_{3})+(\varepsilon\mathbf{x}_{2}+\mathbf{x}%
_{3})\otimes(-\varepsilon\mathbf{x}_{1}+\mathbf{x}_{2})\otimes(\varepsilon
^{-1}\mathbf{x}_{1}+\mathbf{x}_{2})\\
+(\mathbf{x}_{1}+\varepsilon\mathbf{x}_{2})\otimes(\varepsilon\mathbf{x}%
_{3}+\mathbf{x}_{4})\otimes(\varepsilon^{-1}\mathbf{x}_{3}-\mathbf{x}_{4}).
\end{multline*}
It is straight forward to verify that
\[
\lim\nolimits_{\varepsilon\rightarrow0}A_{\varepsilon}=A
\]
where%
\[
A=\mathbf{x}_{1}\otimes\mathbf{x}_{1}\otimes\mathbf{x}_{1}+\mathbf{x}%
_{1}\otimes\mathbf{x}_{3}\otimes\mathbf{x}_{3}+\mathbf{x}_{2}\otimes
\mathbf{x}_{2}\otimes\mathbf{x}_{1}+\mathbf{x}_{2}\otimes\mathbf{x}_{4}%
\otimes\mathbf{x}_{3}+\mathbf{x}_{3}\otimes\mathbf{x}_{2}\otimes\mathbf{x}%
_{2}+\mathbf{x}_{3}\otimes\mathbf{x}_{4}\otimes\mathbf{x}_{4}.
\]
Note that the sequence $A_{\varepsilon}$ exhibits \textsc{parafac} degeneracy:
as $\varepsilon\rightarrow0$, each of the summands becomes unbounded in
magnitude but $A_{\varepsilon}$ remains bounded (and in fact converges to $A$).

Regardless of whether $\mathbf{x}_{1},\mathbf{x}_{2},\mathbf{x}_{3}%
,\mathbf{x}_{4}\in\mathbb{R}^{n}$ or $\mathbb{C}^{n}$, it is clear that for
all $\varepsilon>0$,%
\[
\operatorname*{rank}(A_{\varepsilon})\leq5.
\]
Furthermore, one may show that $\operatorname*{rank}(A)=6$ over $\mathbb{C}$
and therefore $\operatorname*{rank}(A)\geq6$ over $\mathbb{R}$ (cf.\ remarks
at the end of Section \ref{sec:Ten}). In either case, $A$ is an instance in
$\mathbb{R}^{n\times n\times n}$ or $\mathbb{C}^{n\times n\times n}$ where the
approximation problem in \eqref{approx} has no solution for $r=5$ --- since
$\inf\nolimits_{\operatorname*{rank}(X)\leq5}\lVert A-X\rVert=0$ and
$\operatorname*{rank}(A)\geq6$ together imply that%
\[
\operatorname*{argmin}\nolimits_{\operatorname*{rank}(X)\leq5}\lVert
A-X\rVert=\varnothing.
\]
Hence the construction in \cite{BCLR} also yields an explicit example of a
best rank-$r$ approximation problem (over $\mathbb{R}$ and $\mathbb{C}$) that
has no solution.

\section{Existence of globally optimal solution for nonnegative tensor
approximations\label{sec:Exist}}

As we have mentioned in Section \ref{sec:Intro}, nonnegativity constraints are
often natural in the use of \textsc{parafac}. Empirical evidence from Bro's
chemometrics studies revealed that \textsc{parafac} degeneracy was never
observed when fitting nonnegative-valued data with a nonnegative
\textsc{parafac} model. This then led Harshman to conjecture that this is
always the case. The text of his e-mail had been reproduced in \cite{L1}.

The conjectured result involves demonstrating the existence of global minimum
over a non-compact feasible region and is thus not immediate. Nevertheless the
proof is still straightforward by the following observation: If a continuous
real-valued function has a non-empty compact sublevel set, then it has to
attain its infimum --- a consequence of the extreme value theorem. This is
essentially what we will show in the following proof for the nonnegative
\textsc{parafac} loss function (in fact, we will show that all sublevel sets
of the function are compact). We will use the $E$-norm in our proof for
simplicity, the result for other norms then follows from the equivalence of
all norms on finite-dimensional spaces. Essentially the same proof, but in
terms of the more customary $F$-norm, appeared in \cite{L1}. We will follow
the notations in Section \ref{sec:NTD}.

\begin{theorem}
\label{thmE}Let $A\in\mathbb{R}^{d_{1}\times\dots\times d_{k}}$ be
nonnegative. Then
\[
\inf\Bigl\{\left\Vert A-\sum\nolimits_{p=1}^{r}\delta_{p}\,\mathbf{u}%
_{p}\otimes\mathbf{v}_{p}\otimes\dots\otimes\mathbf{z}_{p}\right\Vert
_{E}\Bigm|\boldsymbol{\delta}\in\mathbb{R}_{+}^{r},\mathbf{u}_{p}\in
\Delta^{d_{1}-1},\dots,\mathbf{z}_{p}\in\Delta^{d_{k}-1},p=1,\dots,r\Bigr\}
\]
is attained.
\end{theorem}

\begin{proof}
Recall that $\mathbb{R}_{+}^{n}=\{\mathbf{x}\in\mathbb{R}^{n}\mid
\mathbf{x}\geq0\}$ and $\Delta^{n-1}=\{\mathbf{x}\in\mathbb{R}_{+}^{n}%
\mid\lVert\mathbf{x}\rVert_{1}=1\}$. We define the function $f:\mathbb{R}%
^{r}\times(\mathbb{R}^{d_{1}}\times\dots\times\mathbb{R}^{d_{k}}%
)^{r}\rightarrow\mathbb{R}$ by%
\begin{equation}
f(T):=\left\Vert A-\sum\nolimits_{p=1}^{r}\delta_{p}\,\mathbf{u}_{p}%
\otimes\mathbf{v}_{p}\otimes\dots\otimes\mathbf{z}_{p}\right\Vert _{E}
\label{eq:f}%
\end{equation}
where we let $T=(\delta_{1},\dots,\delta_{r};\mathbf{u}_{1},\mathbf{v}%
_{1},\dots,\mathbf{z}_{1};\dots;\mathbf{u}_{r},\mathbf{v}_{r},\dots
,\mathbf{z}_{r})$ denote the argument of $f$. Let $\mathcal{D}$ be the
following subset of $\mathbb{R}^{r}\times(\mathbb{R}^{d_{1}}\times\dots
\times\mathbb{R}^{d_{k}})^{r}=\mathbb{R}^{r(1+d_{1}+\dots+d_{k})}$,%
\[
\mathcal{D}:=\mathbb{R}_{+}^{r}\times(\Delta^{d_{1}-1}\times\dots\times
\Delta^{d_{k}-1})^{r}.
\]
Note that $\mathcal{D}$ is closed but unbounded. Let the infimum in question
be $\mu:=\inf\{f(T)\mid T\in\mathcal{D}\}$. We will show that the sublevel set
of $f$ restricted to $\mathcal{D}$,%
\[
\mathcal{E}_{\alpha}=\{T\in\mathcal{D}\mid f(T)\leq\alpha\}
\]
is compact for all $\alpha>\mu$ and thus the infimum of $f$ on $\mathcal{D}$
must be attained. The set $\mathcal{E}_{\alpha}=\mathcal{D}\cap f^{-1}%
(-\infty,\alpha]$ is closed since $f$ is continuous (by the continuity of
norm). It remains to show that $\mathcal{E}_{\alpha}$ is bounded. Suppose the
contrary. Then there exists a sequence $(T_{n})_{n=1}^{\infty}\subset
\mathcal{D}$ with $\lVert T_{n}\rVert_{1}\rightarrow\infty$ but $f(T_{n}%
)\leq\alpha$ for all $n$. Clearly, $\lVert T_{n}\rVert_{1}\rightarrow\infty$
implies that $\delta_{q}^{(n)}\rightarrow\infty$ for at least one
$q\in\{1,\dots,r\}$. Note that%
\[
f(T)\geq\Bigl\lvert \lVert A\rVert_{E}-\left\Vert \sum\nolimits_{p=1}%
^{r}\delta_{p}\,\mathbf{u}_{p}\otimes\mathbf{v}_{p}\otimes\dots\otimes
\mathbf{z}_{p}\right\Vert _{E}\Bigr\rvert.
\]
Since all terms involved in the approximant are nonnegative, we have%
\begin{align*}
\left\Vert \sum\nolimits_{p=1}^{r}\delta_{p}\,\mathbf{u}_{p}\otimes
\mathbf{v}_{p}\otimes\dots\otimes\mathbf{z}_{p}\right\Vert _{E}  &
=\sum\nolimits_{j_{1},\dots,j_{k}=1}^{d_{1},\dots,d_{k}}\sum\nolimits_{p=1}%
^{r}\delta_{p}u_{pj_{1}}v_{pj_{2}}\cdots z_{pj_{k}}\\
&  \geq\sum\nolimits_{j_{1},\dots,j_{k}=1}^{d_{1},\dots,d_{k}}\delta
_{q}u_{qj_{1}}v_{qj_{2}}\cdots z_{qj_{k}}\\
&  =\delta_{q}\sum\nolimits_{j_{1},\dots,j_{k}=1}^{d_{1},\dots,d_{k}}%
u_{qj_{1}}v_{qj_{2}}\cdots z_{qj_{k}}\\
&  =\delta_{q}\lVert\mathbf{u}_{q}\otimes\mathbf{v}_{q}\otimes\dots
\otimes\mathbf{z}_{q}\rVert_{E}\\
&  =\delta_{q}\lVert\mathbf{u}_{q}\rVert_{1}\lVert\mathbf{v}_{q}\rVert
_{1}\cdots\lVert\mathbf{z}_{q}\rVert_{1}\\
&  =\delta_{q}%
\end{align*}
where the last two equalities follow from \eqref{norm1} and $\lVert
\mathbf{u}_{q}\rVert_{1}=\lVert\mathbf{v}_{q}\rVert_{1}=\dots=\lVert
\mathbf{z}_{q}\rVert_{1}=1$. Hence, as $\delta_{q}^{(n)}\rightarrow\infty$,
$f(T_{n})\rightarrow\infty$ --- contradicting the assumption that
$f(T_{n})\leq\alpha$ for all $n$.
\end{proof}

The proof essentially shows that the function $f$ is \textit{coercive} --- a
real-valued function $f$ is said to be coercive for minimization if
$\lim_{\lVert\mathbf{x}\rVert\rightarrow+\infty}f(\mathbf{x})=+\infty$
\cite{BT}. This is a standard condition often used to guarantee that a
continuous function on a noncompact domain attains its global minimum and is
equivalent to saying that $f$ has bounded sublevel sets. A minor point to note
is that had we instead optimized over a sum of rank-$1$ terms $\mathbf{u}%
\otimes\mathbf{v}\otimes\dots\otimes\mathbf{z}$, the proof would fail because
the vectors $\mathbf{u},\mathbf{v},\dots,\mathbf{z}$ may be scaled by non-zero
positive scalars that product to $1$, i.e.
\[
\alpha\mathbf{u}\otimes\beta\mathbf{v}\otimes\dots\otimes\zeta\mathbf{z}%
=\mathbf{u}\otimes\mathbf{v}\otimes\dots\otimes\mathbf{z},\qquad\alpha
\beta\cdots\zeta=1.
\]
So for example $(n\mathbf{x})\otimes\mathbf{y}\otimes(\mathbf{z}/n)$ can have
diverging loading factors even while the outer-product remains fixed. We
avoided this by requiring that $\mathbf{u},\mathbf{v},\dots,\mathbf{z}$ be
unit vectors and having a $\boldsymbol{\delta}$ that records the magnitude.

The following proposition provides four useful characterizations of the
statement that the function $\operatorname*{rank}_{+}:\mathbb{R}_{+}%
^{d_{1}\times\dots\times d_{k}}\rightarrow\mathbb{R}$ is upper semicontinuous.
This is the nonnegative rank equivalent of a similar result in \cite{dSL}.

\begin{proposition}
\label{propEquiv}Let $r,k\in\mathbb{N}$ and let the topology on $\mathbb{R}%
_{+}^{d_{1}\times\dots\times d_{k}}$ be induced by the $E$-norm. The following
statements are equivalent; and since the last statement is true by Theorem
\ref{thmE}, so are the others.

\begin{enumerate}
[\upshape(a)]

\item The set $\mathcal{S}_{r}:=\{X\in\mathbb{R}_{+}^{d_{1}\times\dots\times
d_{k}}\mid\operatorname{rank}_{+}(X)\leq r\}$ is closed.

\item Every $A\in\mathbb{R}_{+}^{d_{1}\times\dots\times d_{k}}$,
$\operatorname{rank}_{+}(A)>r$, has a best nonnegative rank-$r$ approximation,
i.e.%
\[
\inf\{\lVert A-X\rVert_{E}\mid\operatorname{rank}_{+}(X)\leq r\}
\]
is attained (by some $X_{r}$ with $\operatorname{rank}_{+}(X_{r})\leq r$).

\item No $A\in\mathbb{R}_{+}^{d_{1}\times\dots\times d_{k}}$,
$\operatorname{rank}_{+}(A)>r$, can be approximated arbitrarily closely by
nonnegative tensors of strictly lower nonnegative rank, i.e.%
\[
\inf\{\lVert A-X\rVert_{E}\mid\operatorname{rank}_{+}(X)\leq r\}>0.
\]

\item No sequence $(X_{n})_{n=1}^{\infty}\subset\mathbb{R}_{+}^{d_{1}%
\times\dots\times d_{k}}$, $\operatorname{rank}_{+}(X_{n})\leq r$, can
converge to $A\in\mathbb{R}_{+}^{d_{1}\times\dots\times d_{k}}$ with
$\operatorname{rank}_{+}(A)>r$.
\end{enumerate}
\end{proposition}

\begin{proof}
(a) $\Rightarrow$ (b): Suppose $\mathcal{S}_{r}$ is closed. Since the set
$\{X\in\mathbb{R}_{+}^{d_{1}\times\dots\times d_{k}}\mid\lVert A-X\rVert
\leq\lVert A\rVert\}$ intersects $\mathcal{S}_{r}$ non-trivially (e.g.\ $0$ is
in both sets). Their intersection $\mathcal{T}=\{X\in\mathbb{R}_{+}%
^{d_{1}\times\dots\times d_{k}}\mid\operatorname{rank}_{+}(X)\leq r,\lVert
A-X\rVert\leq\lVert A\rVert\}$ is a non-empty compact set. Now observe that%
\[
\delta:=\inf\{\lVert A-X\rVert\mid X\in\mathcal{S}_{r}\}=\inf\{\lVert
A-X\rVert\mid X\in\mathcal{T}\}
\]
since any $X^{\prime}\in\mathcal{S}_{r}\backslash\mathcal{T}$ must have
$\lVert A-X^{\prime}\rVert>\lVert A\rVert$ while we know that $\delta
\leq\lVert A\rVert$. By the compactness of $\mathcal{T}$, there exists
$X_{\ast}\in\mathcal{T}$ such that $\lVert A-X_{\ast}\rVert=\delta$. So the
required infimum is attained by $X_{\ast}\in\mathcal{T}\subset\mathcal{S}$.
The remaining implications (b) $\Rightarrow$ (c) $\Rightarrow$ (d)
$\Rightarrow$ (a) are obvious.
\end{proof}

In the language of \cite{BCS, Strass}, this says that `nonnegative border
rank' coincides with nonnegative rank. An immediate corollary is that the
$E$-norm in Theorem \ref{thmE} and Proposition \ref{propEquiv} may be replaced
by any other norm. In fact we will see later that we may replace norms with
more general measures of proximity.

\begin{corollary}
\label{cor:NTA}Let $A=%
\llbracket
a_{j_{1}\cdots j_{k}}%
\rrbracket
\in\mathbb{R}^{d_{1}\times\dots\times d_{k}}$ be nonnegative and
$\lVert\,\cdot\,\rVert:\mathbb{R}^{d_{1}\times\dots\times d_{k}}%
\rightarrow\lbrack0,\infty)$ be an arbitrary norm. Then
\[
\inf\Bigl\{\left\Vert A-\sum\nolimits_{p=1}^{r}\delta_{p}\,\mathbf{u}%
_{p}\otimes\mathbf{v}_{p}\otimes\dots\otimes\mathbf{z}_{p}\right\Vert
\Bigm|\boldsymbol{\delta}\in\mathbb{R}_{+}^{r},\mathbf{u}_{p}\in\Delta
^{d_{1}-1},\dots,\mathbf{z}_{p}\in\Delta^{d_{k}-1},p=1,\dots,r\Bigr\}
\]
is attained.
\end{corollary}

\begin{proof}
This simply follows from the fact that all norms on finite dimensional spaces
are equivalent and so induce the same topology on $\mathbb{R}_{+}^{d_{1}%
\times\dots\times d_{k}}$. So Proposition \ref{propEquiv} holds for any norms.
In particular, the statement (b) in Proposition \ref{propEquiv} for an
arbitrary norm $\lVert\,\cdot\,\rVert$ is exactly the result desired here.
\end{proof}

Corollary~\ref{cor:NTA} implies that the \textsc{parafac} degeneracy discussed
in Section~\ref{SectDegen} does not happen for nonnegative approximations of
nonnegative tensors. There is often a simplistic view of \textsc{parafac}
degeneracy as being synonymous to `between component cancellation' and thus
cannot happen for nonnegative tensor approximation since it is `purely
additive with no cancellation between parts' \cite{LS1, NTF}. While it
provides an approximate intuitive picture, this point of view is flawed since
\textsc{parafac} degeneracy is not the same as `between component
cancellation'. There is cancellation in $n\mathbf{x}\otimes\mathbf{y}%
\otimes\mathbf{z}-(n+\frac{1}{n})\mathbf{x}\otimes\mathbf{y}\otimes\mathbf{z}$
but the sequence exhibits no \textsc{parafac} degeneracy. Conversely, the
sequence of nonnegative tensors%
\[
A_{n}=\left[
\begin{array}
[c]{cc}%
0 & 1\\
1 & 1/n
\end{array}
\right\vert \!\left.
\begin{array}
[c]{cc}%
1 & 1/n\\
1/n & 1/n^{2}%
\end{array}
\right]  \in\mathbb{R}^{2\times2\times2}%
\]
may each be decomposed nonnegatively as%
\begin{equation}
A_{n}=A+\frac{1}{n}B+\frac{1}{n^{2}}C \label{eq2}%
\end{equation}
with $A,B,C\in\mathbb{R}^{2\times2\times2}$ given by%
\[
A=\left[
\begin{array}
[c]{cc}%
0 & 1\\
1 & 0
\end{array}
\right\vert \!\left.
\begin{array}
[c]{cc}%
1 & 0\\
0 & 0
\end{array}
\right]  ,\quad B=\left[
\begin{array}
[c]{cc}%
0 & 0\\
0 & 1
\end{array}
\right\vert \!\left.
\begin{array}
[c]{cc}%
1 & 0\\
1 & 0
\end{array}
\right]  ,\quad C=\left[
\begin{array}
[c]{cc}%
0 & 0\\
0 & 0
\end{array}
\right\vert \!\left.
\begin{array}
[c]{cc}%
0 & 0\\
0 & 1
\end{array}
\right]  ,
\]
and each may in turn be decomposed into a sum of rank-$1$ terms. While there
is no `between component cancellation' among these rank-$1$ summands, it is
known \cite{dSL} that the convergence%
\[
\lim_{n\rightarrow\infty}A_{n}=A
\]
exhibits \textsc{parafac} degeneracy over $\mathbb{R}^{2\times2\times2}$ or
$\mathbb{C}^{2\times2\times2}$, where there are decompositions of $A_{n}$
different from the one given in \eqref{eq2} exhibiting \textsc{parafac}
degeneracy. That such decompositions cannot happen over $\mathbb{R}%
_{+}^{2\times2\times2}$ is precisely the statement of
Proposition~\ref{propEquiv}, which we proved by way of Theorem~\ref{thmE}.

\section{Br\`{e}gman divergences\label{secBreg}}

In many applications, a norm may not be the most suitable measure of
proximity. Other measures based on entropy, margin, spectral separation,
volume, etc, are often used as loss functions in matrix and tensor
approximations. Such measures may not even be a metric, an example being the
Br\`{e}gman divergence \cite{Bre, DT, EOM}, a class of proximity measures that
often have information theoretic or probabilistic interpretations. In the
definition below, $\operatorname{ri}(\Omega)$ denotes the relative interior of
$\Omega$, i.e.\ the interior of $\Omega$ regarded as a subset of its affine
hull; $\lVert\,\cdot\,\rVert$ is any arbitrary norm on $\mathbb{R}%
^{d_{1}\times\dots\times d_{k}}$ --- again the choice of which is immaterial
since all norms induce the same topology on $\Omega$.

\begin{definition}
\label{def:bregman}Let $\varnothing\neq\Omega\subseteq\mathbb{R}^{d_{1}%
\times\dots\times d_{k}}$ be a closed {convex set}. Let $\varphi
:\Omega\rightarrow\mathbb{R}$ be {continuously differentiable }on
$\operatorname{ri}(\Omega)$ and strictly convex and continuous on $\Omega$.
The function $D_{\varphi}:\Omega\times\operatorname{ri}(\Omega)\rightarrow
\mathbb{R}$ defined by%
\[
D_{\varphi}(A,B)=\varphi(A)-\varphi(B)-\langle\nabla\varphi(B),A-B\rangle
\]
is a \textbf{Br\`{e}gman divergence} if

\begin{enumerate}
[\upshape (i)]

\item For any fixed $A\in\Omega$, the sublevel set%
\[
\mathcal{L}_{\alpha}(A)=\{X\in\operatorname{ri}(\Omega)\mid D_{\varphi
}(A,X)\leq\alpha\}
\]
is bounded for all $\alpha\in\mathbb{R}$.

\item Let $(X_{n})_{n=1}^{\infty}\subset\operatorname{ri}(\Omega)$ and
$A\in\Omega$. If%
\[
\lim\nolimits_{n\rightarrow\infty}\lVert A-X_{n}\rVert=0,
\]
then%
\[
\lim\nolimits_{n\rightarrow\infty}D_{\varphi}(A,X_{n})=0.
\]

\item Let $(X_{n})_{n=1}^{\infty}\subset\operatorname{ri}(\Omega)$,
$A\in\Omega$, and $(A_{n})_{n=1}^{\infty}\subset\Omega$. If%
\[
\lim\nolimits_{n\rightarrow\infty}\lVert A-X_{n}\rVert=0,\quad\limsup
\nolimits_{n\rightarrow\infty}\lVert A_{n}\rVert<\infty,\quad\lim
\nolimits_{n\rightarrow\infty}D_{\varphi}(A_{n},X_{n})=0,
\]
then%
\[
\lim\nolimits_{n\rightarrow\infty}\lVert A_{n}-X_{n}\rVert=0.
\]

\end{enumerate}
\end{definition}

Note that $D_{\varphi}(A,B)\geq0$ and that $D_{\varphi}(A,B)=0$ iff $A=B$ by
the strict convexity of $\varphi$. However $D_{\varphi}$ need not satisfy the
triangle inequality nor must it be symmetric in its two arguments. So a
Br\`{e}gman divergence is not a metric in general.

Br\`{e}gman divergences are particularly important in nonnegative matrix and
tensor decompositions \cite{LS1, NTF}. In fact, one of the main novelty of
\textsc{nmf} as introduced by Lee and Seung \cite{LS1} over the earlier
studies in technometrics \cite{CDP, KtB, Paa2} is their use of the
\textit{Kullback-Leibler divergence} \cite{KL} as a proximity
measure\footnote{This brought back memories of the many intense e-mail
exchanges with Harshman, of which one was about the novelty of \textsc{nmf}.
His fervently argued messages will be missed.}. The \textsc{kl} divergence is
defined for nonnegative matrices in \cite{LS1} but it is straightforward to
extend the definition to nonnegative tensors. For $A\in\mathbb{R}_{+}%
^{d_{1}\times\dots\times d_{k}}$ and $B\in\operatorname{ri}(\mathbb{R}%
_{+}^{d_{1}\times\dots\times d_{k}})$, this is%
\[
D_{\operatorname{KL}}(A,B)=\sum\nolimits_{j_{1},\dots,j_{k}=1}^{d_{1}%
,\dots,d_{k}}\Bigl[a_{j_{1}\cdots j_{k}}\log\Bigl(\frac{a_{j_{1}\cdots j_{k}}%
}{b_{j_{1}\cdots j_{k}}}\Bigr)-a_{j_{1}\cdots j_{k}}+b_{j_{1}\cdots j_{k}%
}\Bigr],
\]
where $0\log0$ is taken to be $0$, the limiting value. It comes from the
following choice of $\varphi$,%
\[
\varphi_{\operatorname{KL}}(A)=\sum\nolimits_{j_{1},\dots,j_{k}=1}%
^{d_{1},\dots,d_{k}}a_{j_{1}\cdots j_{k}}\log a_{j_{1}\cdots j_{k}}.
\]
We note that Kullback and Liebler's original definition \cite{KL} was in terms
of probability distributions. The version that we introduced here is a slight
generalization. When $A$ and $B$ are probability distributions as in Section
\ref{sec:NTD}, then $\lVert A\rVert_{E}=\lVert B\rVert_{E}=1$ and our
definition reduces to the original one in \cite{KL},
\[
D_{\operatorname{KL}}(A,B)=\sum\nolimits_{j_{1},\dots,j_{k}=1}^{d_{1}%
,\dots,d_{k}}a_{j_{1}\cdots j_{k}}\log\Bigl(\frac{a_{j_{1}\cdots j_{k}}%
}{b_{j_{1}\cdots j_{k}}}\Bigr).
\]
In this case $D_{\operatorname{KL}}(A,B)$ may also be interpreted as the
relative entropy of the respective distributions.

It is natural to ask if the following analogous nonnegative tensor
approximation problem for Br\`{e}gman divergence will always have a solution:%
\begin{equation}
X_{r}\in\operatorname*{argmin}\{D_{\varphi}\left(  A,X\right)  \mid
X\in\operatorname{ri}(\Omega),\operatorname*{rank}\nolimits_{+}(X)\leq r\}.
\label{prob1}%
\end{equation}
Clearly, the problem cannot be expected to have a solution in general since
$\operatorname{ri}(\Omega)$ is not closed. For example let $A=\mathbf{e}%
\otimes\mathbf{e}\otimes\mathbf{e}\in\mathbb{R}_{+}^{2\times2\times2}$ where
$\mathbf{e}=[1,0]^{\top}$, then%
\[
\inf\{D_{\operatorname{KL}}\left(  A,X\right)  \mid X\in\operatorname{ri}%
(\mathbb{R}_{+}^{2\times2\times2}),\operatorname*{rank}\nolimits_{+}%
(X)\leq1\}=0
\]
cannot be attained by any $\mathbf{x}\otimes\mathbf{y}\otimes\mathbf{z}%
\in\operatorname{ri}(\mathbb{R}_{+}^{2\times2\times2})$ since if we set
$\mathbf{x}_{n}=\mathbf{y}_{n}=\mathbf{z}_{n}=[1,n^{-1}]^{\top}$, then as
$n\rightarrow\infty$,%
\[
D_{\operatorname{KL}}\left(  A,\mathbf{x}_{n}\otimes\mathbf{y}_{n}%
\otimes\mathbf{z}_{n}\right)  =\frac{1}{n^{3}}\rightarrow0.
\]
This is simply a consequence of the way a Br\`{e}gman divergence is defined
and has nothing to do with any peculiarities of tensor rank, unlike the
example discussed in Section~\ref{SectDegen}. This difficulty may be avoided
by posing the problem for any closed (but not necessarily compact) subset of
$\operatorname{ri}(\Omega)$.

\begin{proposition}
\label{prop:bregman}Let $\Omega$ be a closed convex subset of $\mathbb{R}%
_{+}^{d_{1}\times\dots\times d_{k}}$ and $A\in\Omega$. Let $D_{\varphi}%
:\Omega\times\operatorname{ri}(\Omega)\rightarrow\mathbb{R}$ be a Br\`{e}gman
divergence. Then
\begin{equation}
\inf\{D_{\varphi}\left(  A,X\right)  \mid X\in K,\operatorname*{rank}%
\nolimits_{+}(X)\leq r\} \label{eq1}%
\end{equation}
is attained for any closed subset $K\subseteq\operatorname{ri}(\Omega)$.
\end{proposition}

\begin{proof}
Recall that $\mathcal{S}_{r}:=\{X\in\mathbb{R}_{+}^{d_{1}\times\dots\times
d_{k}}\mid\operatorname*{rank}\nolimits_{+}(X)\leq r\}$. The statement is
trivial if $\operatorname*{rank}_{+}(A)\leq r$. So we will also assume that
$\operatorname*{rank}_{+}(A)\geq r+1$. Let $\mu$ be the infimum in \eqref{eq1}
and let $\alpha>\mu$. By (i) in Definition \ref{def:bregman}, the sublevel set
$\mathcal{L}_{\alpha}(A)$ is bounded and so its subset%
\[
K\cap\mathcal{S}_{r}\cap\mathcal{L}_{\alpha}(A)=\{X\in K\cap\mathcal{S}%
_{r}\mid D_{\varphi}(A,X)\leq\alpha\}
\]
must also be bounded. Note that $K\cap\mathcal{S}_{r}$ is closed. Since
$\varphi$ is continuously differentiable on $\operatorname{ri}(\Omega)$, the
function $X\mapsto D_{\varphi}(A,X)$ is continuous and so $K\cap
\mathcal{S}_{r}\cap\mathcal{L}_{\alpha}(A)$ is also closed. Hence $D_{\varphi
}(A,X)$ must attain $\mu$ on the compact set $K\cap\mathcal{S}_{r}%
\cap\mathcal{L}_{\alpha}(A)$.
\end{proof}

As one can see from the proof, Proposition~\ref{prop:bregman} extends to any
other measure of proximity $d(A,X)$ where the function $X\mapsto d(A,X)$ is
continuous and coercive. Of course this is just a restatement of the problem,
the bulk of the work involved is usually to show that the proximity function
in question has those required properties.

\section{Aside: norm-regularized and orthogonal approximations}

We have often been asked about norm-regularized and orthogonal approximations
of tensors that are not necessarily nonnegative. These approximation problems
are useful in practice \cite{Como94:SP, Ha2, Paa}. Nevertheless these always
have optimal solutions for a much simpler reason --- they are continuous
optimization problems over compact feasible set, so the existence of a global
minima is immediate from the extreme value theorem (note that this is not the
case for nonnegative tensor approximation). In the following, we will let
$A\in\mathbb{R}^{d_{1}\times\dots\times d_{r}}$, not necessarily nonnegative.

Recall that $\operatorname*{O}(n,r)$, the set of $n\times r$ matrices ($r\leq
n$) with orthonormal columns, is compact in $\mathbb{R}^{n\times r}$. If we
impose orthonormality constraints on the normalized loading factors in
\eqref{parafac}, i.e.\ $[\mathbf{u}_{1},\dots,\mathbf{u}_{r}]\in
\operatorname*{O}(d_{1},r),\dots,[\mathbf{z}_{1},\dots,\mathbf{z}_{r}%
]\in\operatorname*{O}(d_{k},r)$, then it follows that $\lvert\lambda_{p}%
\rvert\leq\lVert A\rVert_{F}^{2}$ for $p=1,\dots,r$,
i.e.\ $\boldsymbol{\lambda}\in\lbrack-\lVert A\rVert_{F},\lVert A\rVert
_{F}]^{r}\subset\mathbb{R}^{r}$. Since the \textsc{parafac} objective is
continuous and we are effectively minimizing over the compact feasible region%
\[
\lbrack-\lVert A\rVert_{F},\lVert A\rVert_{F}]^{r}\times\operatorname*{O}%
(d_{1},r)\times\dots\times\operatorname*{O}(d_{k},r),
\]
this shows that orthogonal \textsc{parafac} always has a globally optimal solution.

Next, the regularization proposed in \cite{Paa} is to add to the
\textsc{parafac} objective terms proportional to the $2$-norm of each loading
factor, i.e.%
\begin{equation}
\left\Vert A-\sum\nolimits_{p=1}^{r}\mathbf{a}_{p}\otimes\mathbf{b}_{p}%
\otimes\dots\otimes\mathbf{c}_{p}\right\Vert _{F}^{2}+\rho\sum\nolimits_{p=1}%
^{r}(\lVert\mathbf{a}_{p}\rVert_{2}^{2}+\lVert\mathbf{b}_{p}\rVert_{2}%
^{2}+\dots+\lVert\mathbf{c}_{p}\rVert_{2}^{2}). \label{reg}%
\end{equation}
From constrained optimization theory, we know that, under some regularity
conditions, minimizing a continuous function $f(\mathbf{x}_{1},\dots
,\mathbf{x}_{k})$ under constraints $\lVert\mathbf{x}_{i}\rVert_{2}=r_{i}$,
$i=1,\dots,k,$ is \textit{equivalent} to minimizing the functional
$f(\mathbf{x}_{1},\dots,\mathbf{x}_{k})+\sum_{i=1}^{k}\rho_{i}\lVert
\mathbf{x}_{i}\rVert_{2}^{2}$ for appropriate $\rho_{1},\dots,\rho_{k}%
\in\mathbb{R}$. In a finite dimensional space, the sphere of radius $r_{i}$ is
compact, and so is the feasible set defined by $\lVert\mathbf{x}_{i}\rVert
_{2}=r_{i}$, $i=1,\dots,k$, and thus $f$ must attain its extrema. In the same
vein, \eqref{reg} is equivalent to an equality constrained optimization
problem and so norm-regularized \textsc{parafac} always has a globally optimal
solution. This approach may also be applied to regularizations other than the
one discussed here.

\section{Acknowledgements}

We thank the reviewers for their many helpful comments.

\end{document}